\documentclass[10pt]{article}
\usepackage[utf8]{inputenc}
\usepackage{mathtools,
        mathrsfs,amssymb,amsthm, amsmath,
        bm}
\usepackage{comment}
\usepackage{hyperref}
\usepackage{url}
\usepackage[a4paper, left = 1in, right = 1in, top =1in, bottom = 1in]{geometry} 
\usepackage{ dsfont }   
\usepackage{enumitem}
\usepackage[british]{babel}

\newtheorem{thm}{Theorem}
\newtheorem{pro}{Proposition}

\theoremstyle{definition}
\newtheorem*{rem}{Remark}

\makeatletter
\renewenvironment{proof}[1][\proofname] {\par\pushQED{\qed}\normalfont\topsep6\p@\@plus6\p@\relax\trivlist\item[\hskip\labelsep\bfseries#1\@addpunct{.}]\ignorespaces}{\popQED\endtrivlist\@endpefalse}
\makeatother

\def\[#1\]{\begin{align*}#1\end{align*}}

\newcommand{\R}{\mathbb{R}}

\newcommand{\N}{\mathbb{N}}
\newcommand{\ceq}{\coloneqq}

\newcommand{\E}{\mathbb{E}}
\renewcommand{\P}{\mathbb{P}}

\newcommand{\lrtx}[1]{\ \text{#1} \ }
\newcommand{\ltx}[1]{\ \text{#1}}

\newcommand{\I}{\mathds{1}}
\newcommand{\df}{\mathop{}\!\mathrm{d}}

\newcommand{\D}{\mathbb{D}}

\newcommand{\F}
{\mathscr{F}}
\newcommand{\lb}{\lbrack}
\newcommand{\rb}
{\rbrack}
\newcommand{\id}{\mathrm{id}}

\begin{document}
\title{Tail Probability and Divergent Series}
\author{Yu-Lin Chou\thanks{Yu-Lin Chou, Institute of Statistics, National Tsing Hua University, Hsinchu 30013, Taiwan,  R.O.C.; Email: \protect\url{y.l.chou@gapp.nthu.edu.tw}.}}
\date{}
\maketitle

\begin{abstract}
From mostly a measure-theoretic consideration, we show that for every nonnegative, finite, and $L^{1}$ function on a given finite measure space there is some nontrivial sequence of real numbers such that the series, obtained from summing over the term-by-term products of the reals and the summands of any divergent series with positive, vanishing summands such as the harmonic series, is convergent and no greater than the integral of the function. In terms of inequalities, the implications add additional information on mathematical expectation and the behavior of divergent series with positive, vanishing summands, and establish in a broad sense some new, unexpected connections between probability theory and, for instance, number theory.\\


            
\end{abstract}

\section{Introduction}

If $a \equiv (a_{n})_{n \in \N}$ is a sequence of positive real numbers,
\textit{i.e.} a sequence of reals $> 0$, let $H_{n}(a) \ceq \sum_{j=1}^{n}a_{j}^{-1}$
for all $n \in \N$; thus $a_{n} \ceq n$ for every $n \in \N$ implies
that each $H_{n}(a)$ is the $n$-th harmonic number. Take any sum
of the form $\sum_{n=1}^{0}$ to be $=0$, and write $H_{n}$ for
the $n$-th harmonic number for each $n \in \N$ ; then, for every
$x$ in the set $\R_{+}$ of all reals $\geq 0$, we have $\sum_{n=1}^{\lfloor x \rfloor}n^{-1}\I_{\lb H_{n}, +\infty \lb }(x) \leq x,$
the right-hand side of which is a luxury upper bound for $H_{\lfloor x \rfloor}$.
Indeed, the presence of the indicators further allows us to write
\[
\sum_{n \in \N}n^{-1}\I_{\lb H_{n}, +\infty \lb }(x) \leq x
\]for every $x \in \R_{+}$. If $\D^{x}$ denotes the Dirac measure
$A \mapsto \I_{A}(x)$ over $\R$ for every $x \in \R$, then the
above inequality holds for every $x \in \R_{+}$ if and only if \[
\sum_{n \in \N}n^{-1}\D^{x}([H_{n}, +\infty[) \leq \int_{\R} y \df \D^{x}(y)
\]for every $x \in \R_{+}$. Since, in a probabilistic interpretation,
a Dirac measure (restricted to the Borel sigma-algebra) over $\R$
is a probability distribution, called a degenerate distribution, over
$\R$, writing $\P$ for any given $\D^{x}$ with $x \in \R^{+}$
and denoting by $\id$ the identity automorphism on $\R_{+}$ imply
that \[
\sum_{n \in \N}n^{-1}\P(\lb H_{n}, +\infty \lb ) = \sum_{n \in \N}n^{-1}\P(\id \geq H_{n}) \leq \E \id.
\]Here (and throughout) $\E$ denotes the integration operator with
respect to the in-context underlying measure. 

Now, given any finite measure space $(\Omega, \F, \P)$, can we obtain
the last inequality above for every function in $L^{1}(\Omega)$ with
values in $\R_{+}$ under the mere additional assumptions that $(a_{n}^{-1})_{n \in \N}$
is vanishing and that $\sum_{n \in \N}a_{n}^{-1}$ diverges? Although
we will provide a simple proof that it is indeed affirmative, the
decision is not immediate as for every $X \in L^{1}(\Omega, \R_{+})$
we have \[
a_{n}^{-1}\P (X \geq H_{n}(a)) \leq a_{n}^{-1}\P (X \geq a_{n}^{-1}) \leq \E X
\]for every $n \in \N$. Since $\sum_{n \in \N}a_{n}^{-1}$ goes beyond
every bound, and since $X$ is in a sense quite arbitrary, the behavior
of the products present in the above inequalities is not \textit{a priori}
clear; thus it is not immediate regarding where summing over the products
would lead to, let alone asserting some relations between those series
and the integral. 

Since a probability measure is simply a suitably scaled finite measure,
throughout we will argue in terms of a probabilistically-inclined
language; our verbal narration then gets more flexible without loss
of legitimacy. For instance, we may now at will refer to the real
number $\D^{x}(\lb H_{n}, +\infty \lb )$ as a tail probability of
the random variable $\id$ and the integral $\int_{\R} y \df \D^{x}(y)$
as the expectation of $\id$. At the same time, the results are never
limited to the realm of probability theory. We will prove

\begin{thm}\label{thm1}

Let $(\Omega, \F, \P)$ be a probability space; let $X \in L^{1}(\Omega, \R_{+})$;
let $a$ be a sequence of positive reals whose reciprocals converge
to $0$; let $\sum_{n \in \N}a_{n}^{-1}$ be divergent. Then \[
\sum_{n \in \N}a_{n}^{-1}\P (X \geq H_{n}(a)) &\leq \E X \leq \sum_{n \in \N}a_{n}^{-1}\P(X \geq a_{n}^{-1});\\
\sum_{n \in \N}a_{n}^{-1}\P(X \geq H_{n}(a)) &\leq 1 + \sum_{n \in \N}\P(X \geq n);\\
\sum_{n \in \N}a_{n}^{-1}\P(X \geq a_{n}^{-1}) &\geq \sum_{n \in \N}\P(X \geq n). \tag*{$\qed$}
\]

\end{thm}

This result greatly generalizes the previous intuitive observation.
Indeed, as we will illustrate, Theorem \ref{thm1} admits some interesting
implications regarding results in number theory and probability theory. 

The proof, to a great extent, depends on representing a measurable
$\R_{+}$-valued function as a (convergent) series whose summands
are the term-by-term products of some vanishing sequence of positive
reals that forms a divergent series and some sequence of (measurable)
indicators. Fortunately, a short, insightful, and elegant proof for
the indicated representation to be possible is ``almost'' known. 

The next section, Section 2, elaborates on the justification and presents
some intuitive discussions.

\section{Proof and Remarks}

Throughout this article, we fix a probability space $(\Omega, \F, \P)$.
As we argue in terms of a probabilistic language, we adopt some conventional
notation rules in probability theory, which is also convenient for
our purposes without costing clarity. If $X$ is a map defined on
$\Omega$, a set of the form $\{ X \lrtx{has a property} P \}$ means
$\{ \omega \in \Omega \mid X(\omega) \ltx{has the property} \}$;
when written next to the measure $\P$, such a set will take the form
$(X \lrtx{has the property} P)$. The notation rules are ``reasonably''
natural once we see that the form of the properties $P$ concerning
probability theory is usually complicated, and our case is not an
exception. 

It can be shown that for every $\F$-measurable $X: \Omega \to \R_{+}$
there are some $A_{1}, A_{2}, \dots \in \F$ such that $X = \sum_{n \in \N}n^{-1}\I_{A_{n}}$;
Evans and Gariepy \cite{eg} provides an elementary, short, and elegant
proof\footnote{It may be worthwhile to point out here that Evans and Gariepy \cite{eg} develops their  materials with respect to the Carath{\'e}odory paradigm. The approach that is (probably) more common belongs to the Radon paradigm. Besides, their proof applies under conditions that are more general in a certain direction; we slightly generalize the proof in another direction.} ,
which may be found under Theorem 1.12. But the argument is also ready for a slight generalization beyond considering the “harmonic coefficients” $n^{-1}$, which is in fact recorded (in the sense of Footnote 1) under Theorem 2.3.3 in Federer \cite{f} with a one-line proof sketch. For clarity and for both the reader's and our later reference, we shall make the slightly generalized argument, based on what is given Evans and Gariepy \cite{eg}, enter the following proof of Theorem \ref{thm1}:

\begin{proof}[Proof of Theorem \ref{thm1}]

Let $A_{1} \ceq \{ X \geq a_{1}^{-1} \}$; let $A_{n} \ceq \{ X \geq a_{n}^{-1} + \sum_{j=1}^{n-1}a_{j}^{-1}\I_{A_{j}} \}$
for all $n \geq 2$ by induction. Then $X \geq a_{1}^{-1}\I_{A_{1}}$
on $\Omega$. If there is some $n \in \N$ such that $X \geq \sum_{j=1}^{n}a_{j}^{-1}\I_{A_{j}}$
on $\Omega$, and if $X < \sum_{j=1}^{n+1}a_{j}^{-1}\I_{A_{j}}$ on
$\Omega$, then $\I_{A_{n+1}} > 0$ and hence $X \geq a_{n+1}^{-1} + \sum_{j=1}^{n}a_{j}^{-1}\I_{A_{j}} \geq \sum_{j=1}^{n+1}a_{j}^{-1}\I_{A_{j}}$
on $\Omega$, a contradiction. Since $X \geq \sum_{j=1}^{n}a_{j}^{-1}\I_{A_{j}}$
on $\Omega$ for all $n \in \N$, and since $X$ is finite everywhere
by assumption, we have \[
\sum_{n \in \N}a_{n}^{-1}\I_{A_{n}} < +\infty
\]on $\Omega$. Since $\sum_{n \in \N}a_{n}^{-1}$ is divergent by assumption,
for every $\omega \in \Omega$ there are infinitely many $n \in \N$
such that $\I_{A_{n}}(\omega) = 0$. It then follows that \[
0 \leq X(\omega) - \sum_{j=1}^{n}a_{j}^{-1}\I_{A_{j}}(\omega) < a_{n+1}^{-1} 
\]for all $\omega \in \Omega$ and for infinitely many $n \in \N$.
But the sequence $a_{1}^{-1}, a_{2}^{-1}, \dots$ is vanishing by
assumption, we have \[
X = \sum_{n \in \N}a_{n}^{-1}\I_{A_{n}}.
\]

Since $X$ is $L^{1}$ by assumption, the monotone convergence theorem
implies that \[
\E X = \sum_{n \in \N}a_{n}^{-1}\P (A_{n}) = \sum_{n \in \N}a_{n}^{-1}\P \bigg( X \geq a_{n}^{-1} + \sum_{j=1}^{n-1}a_{j}^{-1}\I_{A_{j}} \bigg) < +\infty.
\]Upon observing that $\P (X \geq H_{n}(a)) \leq \P (A_{n}) \leq \P (X \geq a_{n}^{-1})$
for all $n \in \N$, the first two inequalities pertaining to $\E X$
follow.

To shorten the argument, we refer the reader to Theorem 3.2.1 in Chung
\cite{c}; the result asserts, with a simple proof from a consideration
over the measurable sets $\{ n \leq Y < n+1 \}$ where $Y: \Omega \to \R_{+}$
is measurable-$\F$ and $n \in \N$, that \[
\sum_{n \in \N}\P (Y \geq n) \leq \E Y \leq 1 + \sum_{n \in \N}\P (Y \geq n)
\]for all $\F$-measurable $Y: \Omega \to \R_{+}$. Since $X \in L^{1}(\Omega, \R_{+})$
by assumption, the remaining inequalities follow.\end{proof}

\begin{rem}

{ \ }

\begin{itemize}[leftmargin=*]

\item Theorem \ref{thm1} is not probability-specific; the first
part of the argument above apparently applies to any finite measure,
and the second part is only subject to a suitable replacement of the
constant $1$ with the full finite measure of $\Omega$. 

\item Theorem \ref{thm1} is connected with various familiar divergent
series according as $a_{n} \ceq n$ or $\ceq n^{\delta}$ where $0 < \delta < 1$
is given, or $\ceq n\log n$, or $\ceq p_{n}$ where $p_{n}$ is the
$n$-th prime for each $n \in \N$; they all have positive components
such that their reciprocals form a vanishing sequence and a divergent
series. 

\item Although the harmonic series diverges, there is a nontrivial
way to ``stabilize'' its growth: Take any $L^{1}$, nonnegative
random variable $X$; find its tail probabilities of the form $\P (X \geq H_{n})$;
and take the sum of the products $n^{-1}\P(X \geq H_{n})$ over all
$n \in \N$. More unexpectedly, this procedure applies to divergent
series that grow in a much slower way; for instance, consider the
harmonic series of primes. In another sense, the procedure also provides
a nontrivial, measure-theoretic way to construct a convergent series
out of divergent series with positive, vanishing summands. 

\item As stated in the introduction, we have $a_{n}^{-1}\P(X \geq H_{n}(a)) \leq a_{n}^{-1}\P (X \geq a_{n}^{-1}) \leq \E X$
(in particular) for every $L^{1}$ random variable $X$, which prevents
a direct deduction for the behavior of the series formed with respect
to the first two terms. But the representation theorem of measurable,
$\R_{+}$-valued functions furnishes an assertion on the behavior
of the series in terms of definite inequalities.

\item Measurable functions are usually connected with series indirectly
by (partial) representations such as $\sum_{j=1}^{n2^{n}-1}j2^{-n}\I_{\lb j2^{-n}, (j+1)2^{-n} \lb } \circ f + n\I_{ \lb n, +\infty \lb } \circ f$.
\qed

\end{itemize}

\end{rem}

Under the assumptions of Theorem \ref{thm1}, we can further sharpen
the inequality secondly displayed in the statement of Theorem \ref{thm1}
provided that the points with positive probability in the range of
the random variable under consideration is suitably restricted:

\begin{pro}

Let the assumptions of Theorem \ref{thm1} take place. If, in addition,
there is some subset of $\N$ such that $\P (X = n) > 0$ for all
elements $n$ of the subset and the subset has measure $1$ with respect
to the induced measure of $\P$ by $X$, then

\[
\sum_{n \in \N}\P (X \geq n) \geq \sum_{n \in \N}a_{n}^{-1}\P (X \geq H_{n}(a)).
\]

\end{pro}

\begin{proof}

Since $X$ is finite by assumption, for every $\omega \in \Omega$
we have \[
X(\omega) = \int_{0}^{+\infty}\I_{\lb t, +\infty \lb} (X(\omega)) \df t.
\]But $X$ is also $L^{1}$ by assumption, an application of the monotone
convergence theorem and the Fubini's theorem gives \[
\E X = \int_{0}^{+\infty}\P (X \geq t) \df t.
\]By the assumed additional regularity of $X$, taking a partition of
$\rb 0, +\infty \lb$ by intervals such as $\{ \rb j, j+1 \rb \mid j = 0, 1, \dots \}$
implies \[
\E X = \sum_{n \in \N}\P(X \geq n);
\]the desired inequality then follows from Theorem \ref{thm1}. \end{proof}


\begin{thebibliography}{1} 
\bibitem{c} Chung, K.-L. (2000). \textit{A Course in Probability Theory}, third edition.  Academic Press.

\bibitem{eg} Evans, L. C. and Gariepy, R. F. (2015). 
\textit{Measure Theory and Fine Properties of Functions}, first edition. Chapman {\&} Hall.

\bibitem{f} Federer, H. (1996).
\textit{Geometric Measure Theory}, reprint of the first edition. Springer.
\end{thebibliography}
\end{document}